\documentclass[11pt]{amsart}
\usepackage[T1]{fontenc}
\usepackage[latin9]{inputenc}
\usepackage[a4paper]{geometry}
\usepackage{amsmath}
\usepackage{amsthm}
\usepackage{amssymb}
\usepackage{tikz-cd}

\begin{document}

\makeatletter

\newtheorem{thm}{Theorem}[section]
\newtheorem{defn}[thm]{Definition}
\newtheorem{lmma}[thm]{Lemma}
\newtheorem{coro}[thm]{Corollary}

\newcommand{\diam}{\mathrm{diam\,}}

\title{Completion of premetric spaces}

\author{Jos\'e Andr\'es Quintero \and Carlos Uzc\'ategui}
\address{\llap{*\,} Departamento de Matem\'aticas, Facultad de Ciencias, Universidad Nacional de Colombia, Carrera 30 Calle 45, Bogot\'a, COLOMBIA}
\email{joquinteroc@unal.edu.co}

\address{\llap{**\,}Escuela de Matem\'aticas, Facultad de Ciencias, Universidad Industrial de Santander, Ciudad Universitaria, Carrera 27 Calle 9, Bucaramanga, Santander, A.A. 678, COLOMBIA.}
\email{cuzcatea@saber.uis.edu.co}

\keywords{premetric spaces, completion. }

\maketitle

\begin{abstract}
We present a method for completing a premetric space, in the sense introduced by F. Richman in the context of constructive Mathematics without countable choice.
\end{abstract}

\section{Introduction}

F. Richman \cite{Richman2008}  addressed the problem of completing a metric space in constructive mathematics without the axiom of countable choice.  The classical completion based on Cauchy sequences cannot be done in constructive mathematics as it was shown by Lubarsky \cite{Lubarsky2007}. Richman replaced the notion of a metric space (which suppose the existence of $\mathbb{R}$) by a structure he called a  premetric space which only needs the rational numbers. In this paper we present a method for completing a premetric space quite similar to Richman's construction but simpler. 
 
Now we recall the basic definitions and the proposed completion introduced by Richman.

Let $X$ be a nonempty set. A binary relation $d$ between $X\times X$
and the nonnegative rationals is a\textit{ premetric} on $X$, and
we write $d(x,y)\leq q$ instead of $((x,y),q)\in d$, if it satisfies
the following conditions for all $x,y,z\in X$ and all nonnegative rational numbers $p, q$:
\begin{enumerate}
\item $d(x,y)\leq0$ if and only if $x=y$.
\item if $d(x,y)\leq q$ then $d(y,x)\leq q$.
\item if $d(x,z)\leq p$ and $d(z,y)\leq q$ then $d(x,y)\leq p+q$ (triangular inequality).
\item $d(x,y)\leq p$ if and only if $d(x,y)\leq q$ for all $q>p$ (upper continuity).
\end{enumerate}
A set with a premetric is called a \textit{premetric space}, and we
use the notation $(X,d)$ to specify the set and its premetric. Also,
we write $d(x,y)\nleq q$ instead of $((x,y),q)\notin d$. When a relation $d$ satisfies 2., 3., 4. above and also  $d(x,x)\leq 0$ for all $x\in X$, then $d$ is called a {\em pseudo-premetric}. When $(X,d)$ is a pseudo-premetric space, then the relation $x\sim y$ if $d(x,y)\leq 0$ is an equivalence relation. Thus we can define the corresponding quotient and get a premetric space where the induced premetric is given by $d([x],[y])\leq q$ if $d(x,y)\leq q$. The proof that $\sim $ is an equivalence relation, and that the induced premetric is indeed a premetric depends only on the first three defining conditions of a pseudo-premetric.

In general, common notions defined for metric spaces can be copied on premetric spaces. If $D\subseteq X$ and $X$ is premetric space, then $D$ is said to be {\em dense} in $X$ when given any $\varepsilon>0$ and $x\in X$, there exists $y\in D$ such that $d(y,x)\leq \varepsilon$. A map $f:X\to Y$ between two premetric spaces $(X,d_X)$ and $(Y,d_Y)$ is an {\em isometric embedding} if for all $x,x'\in X$ and $q$ a non negative rational,  $d_X(x,x')\leq q$ if, and only if, $d_Y(f(x),f(x'))\leq q$.  When $f$ is in addition onto, then it is called an {\em isometry}. The metric notion of  diameter of a set is also defined on a premetric space as a binary relation, namely, we write $\diam A\leq q$ for $A\subseteq X$ and a nonnegative rational number $q$, if for all $x,y\in A$, $d(x,y)\leq q$.

Richman introduced the following notions. A family $\{S_q:\; q\in \mathbb{Q}^+\}$ of subsets of $X$ is {\em regular} if $d(x,y)\leq p+q$ for all $x\in S_p$ and $y\in S_q$.  Two regular families $S=\{S_q:\; q\in \mathbb{Q}^+\}$  and $T=\{t_q:\; q\in \mathbb{Q}^+\}$  are {\em equivalent} if $d(x,y)\leq p+q$ for all $x\in S_p$ and $y\in T_q$. Let $\widehat{X}_R$ be the quotient of all regular families under that equivalence relation. The  natural identification $i_R:X\to\widehat{X}_R$ is defined by $i_R(x)=[S^x]$, where $S^x=\{S_q:\; q\in \mathbb{Q}^+\}$ with  $S_q=\{x\}$ for all $q\in \mathbb{Q}^+$. Thus we have the notion of completeness introduced by Richman: A premetric space $(X,d)$ is {\em R-complete} if the map  $i_R$ is onto. And finally, a premetric on $\widehat{X}_R$ is defined as follows:
$\widehat{d}_R([S],[T])\leq q$, if for all $\varepsilon >0$, there are $a,b,c\in \mathbb{Q}^+$ and  $s\in S_a$,  $t\in T_b$ such that $a+b+c < q+\varepsilon$ and $d(s,t) \leq c$. Then Richman proved that $i_R:X\rightarrow\widehat{X}_R$ is an isometric embedding of $X$ into $\widehat{X}_R$ with dense image. On his  review  of  Richman's paper \cite{Richman2008}, A. Setzer \cite{Setzer}  remarked  that  there was not a formal verification that  $(\widehat{X}_R, \widehat{d}_R)$ was indeed a R-complete premetric space. Motivated in part for Setzer's remark, we define a premetric space $\widehat{X}$ using,  instead of regular families, some collections of subsets of $X$ similar to Cauchy filters but simpler. We introduce a notion of completeness and show that $\widehat{X}$ is a completion of $X$.  Moreover, we show that  $(\widehat{X}_R, \widehat{d}_R)$ is isometric to ours and it is R-complete.

\section{Completeness}

In order to simplify the treatment of completeness developed by Richman, we use, instead of regular families, a more abstract notion similar to that of a Cauchy filter.

\begin{defn}
Let $X$ be a premetric space and $F\subseteq\mathcal{P}(X)$. We
say that $F$ is a \textit{Cauchy family} on $X$ if it satisfies
the following:
\end{defn}

\begin{itemize}
\item[\textit{i})] $S\cap T\neq\emptyset$ for all $S,T\in F$.
\item[\textit{ii})] for all $\varepsilon>0$ there exists $S\in F$ such that $\diam S\leq\varepsilon$.
\end{itemize}

A trivial example of Cauchy family is  $\{\{x\}\}$, where $x$ is any point of a premetric space. A Cauchy filter is a just a Cauchy family which is also a filter. Notice that a Cauchy family does not need to be even a filter base, as it could fail to have the finite intersection property. For instance, consider  $\mathbb{Q}$ as a premetric space with its natural premetric given by the relation $\left|x-y\right|\leq q$. Let $F=\{S_q:\;q\in \mathbb{Q}^+\}\cup \{\mathbb{Q}^+_0, \mathbb{Q}^-_0\}$, where $\mathbb{Q}^+_0=\mathbb{Q}^+\cup\{0\}$, $\mathbb{Q}^-_0=\mathbb{Q}^-\cup \{0\}$ and 
$$
S_q=\{t\in \mathbb{Q}:0<\left|t\right|\leq q\}
$$
for $q\in \mathbb{Q}^+$. Then $F$ is a Cauchy family that does not satisfy the finite intersection property as  $S_p\cap \mathbb{Q}^+_0\cap\mathbb{Q}^-_0=\emptyset$. 

There is no problem in using Cauchy filters instead of regular families, but the notion of a Cauchy family is sufficient for our purposes. 

Now we introduce a pseudo-premetric on the collection of  Cauchy families. 

\begin{defn}
\label{def-CauchyF}
Let $F$ and $F'$ be Cauchy families on $X$, we say that $d(F,F')\leq q$,
if for all $\varepsilon>0$ there exist $S\in F$ and $T\in F'$
such that:

\begin{enumerate}
\item[\textit{i})] $\mathrm{diam}\,S,\mathrm{diam}\,T\leq\varepsilon.$
\item[\textit{ii})] $\mathrm{diam}(S\cup T)\leq q+\varepsilon.$
\end{enumerate}
\end{defn}
The intuition behind the premetric on the collection of Cauchy families is as follows. Considering each Cauchy family $F$ as a point, then the elements of $F$ are  ``good approximations'' of $F$. Thus smaller the diameter of an element $S\in F$, the  better approximation it is. Thus, roughly speaking,  two Cauchy families $F$, $F'$ are at distance at most $q$, when there are approximations $S$ and $T$ of $F$ and $F'$, respectively, such that $S$ and $T$ are at distance at most $q$.

Here we abuse a bit with the notation, we use $d$ for both the premetric on $X$ and the  ternary relation just defined above. 

\begin{thm}
Let $(X,d)$ be a premetric space. The  relation $d$ on the collection of Cauchy families is a pseudo-premetric.
\end{thm}

\begin{proof}
Let $F$ and $F'$ be two Cauchy families and $q$ a nonnegative rational number. Applying directly conditions $i)$ and $ii)$ from  Definition \ref{def-CauchyF} we have that $d(F,F)\leq q$  and $d(F,F')\leq q$ if $d(F',F)\leq q$.

To prove the triangular inequality, suppose that  $d(F,G)\leq p$ and $d(G,H)\leq q$. Then for each $\varepsilon >0$ there are subsets $S\in F$, $U\in H$ and $T,T'\in G$ such that $\diam S$, $\diam U$, $\diam T$, $\diam T'\leq\frac{\varepsilon}{2}$, and such that $\diam (S\cup T)\leq p+\frac{\varepsilon}{2}$ and $\diam (T'\cup U)\leq q+\frac{\varepsilon}{2}$. By condition $i)$ of the definition of  a Cauchy family, there is  $t\in T\cap T'$. Thus $d(x,t)\leq p+\frac{\varepsilon}{2}$ and $d(t,y)\leq q+\frac{\varepsilon}{2}$ for all $x\in S$ and $y\in U$, so by  the triangular inequality on $X$ we get $d(x,y)\leq p+q+\varepsilon$, and $\diam (S\cup U)\leq p+q+\varepsilon$. Also, by upper continuity on $X$, $\diam S,\diam U\leq \varepsilon$, so we have shown that $d(F,H)\leq p+q$.

Finally to show upper continuity, suppose first that $d(F,F')\leq p$ and $q>p$. Then for each  $\varepsilon >0$ we can find $S\in F$ and $T\in F'$ such that $\diam S,\diam T\leq \varepsilon$, satisfying moreover that $\diam (S\cup T)\leq p+\varepsilon$. Since $p+\varepsilon<q+\varepsilon$ the upper continuity of $d$ on $X$ implies that also $\diam (S\cup T)\leq q+\varepsilon$, and
thus by definition, $d(F,F')\leq q$. Conversely,  suppose $d(F,F')\leq q$ for all $q>p$.  Let $\varepsilon>0$ and set $q=p+\frac{\varepsilon}{2}$, there are $S\in F$ and $T\in F'$ such that $\diam S,\diam T\leq \frac{\varepsilon}{2}$ and 
such that $\diam (S\cup T)\leq q+\frac{\varepsilon}{2}=p+\varepsilon$.  By applying again upper continuity (on $X$) we have that $\diam S,\diam T\leq \varepsilon$ and thus $d(F,F')\leq p$. 
\end{proof}

Now we introduce the premetric space of Cauchy families of a given premetric space. 

\begin{defn}
Let $X$ be a premetric space. The premetric space obtained from the quotient of the set of all Cauchy families on $X$, induced by its respective pseudo-premetric, is denoted by $(\widehat{X},\widehat{d}\,)$. Also, we use the notation $x^*$ for the equivalence class of the family $\{\{x\}\}$.
\end{defn}

As usual,  we denote the elements of $\widehat{X}$ as $[F]$, the equivalence class of the Cauchy family $F$. 

\begin{thm}
\label{Embedded-thm}
Let $X$ be a premetric space. The function $i:X\rightarrow\widehat{X}$ defined as $i(x)=x^*$ is an isometric embedding with dense image.\end{thm}
\begin{proof}
To verify that $i$ is an isometric embedding, suppose first $x,y\in X$ and $d(x,y)\leq q$. Since $\diam (\{x\}\cup\{y\})\leq q\leq q+\varepsilon$ for each $\varepsilon>0$, then $d(\{\{x\}\},\{\{y\}\})\leq q$ and thus $\widehat{d}(x^*,y^*)\leq q$. On the other hand, suppose $\widehat{d}(x^*,y^*)\leq q$, then $\diam (\{x\}\cup\{y\})\leq q+\varepsilon$  for each $\varepsilon>0$, and therefore $d(x,y)\leq q+\varepsilon$ for each $\varepsilon>0$. So by upper continuity we conclude $d(x,y)\leq q$.

Now let us show that $i(X)$ is dense in $\widehat{X}$. Suppose $\varepsilon>0$ and $[F]\in \widehat{X}$. By the definition of a Cauchy family, there exists  $S\in F$ such that $\diam S\leq \varepsilon$. Since $S\not=\emptyset$, fix an element $s\in S$. We claim that $d(F,\{\{s\}\})\leq\varepsilon$. In fact,  let $\delta>0$,  there exists $T\in F$ such that $\diam T\leq \delta$. Let $t\in T\cap S$, as $S\cap T\not=\emptyset$. Then for all $u\in T$ we have that $d(u,t)\leq \delta$ and $d(t,s)\leq \varepsilon$, and therefore $d(u,s)\leq\varepsilon+\delta$. That is, $\diam (T\cup\{\{s\}\})\leq \varepsilon+\delta$ and $d(F,\{\{s\}\})\leq\varepsilon$. Thus we have found an element $s\in X$ such that $\widehat{d}([F],s^*)\leq \varepsilon$.
\end{proof}

Following Richman, we now introduce  the notion of a complete premetric space, in terms of its canonical map. 

\begin{defn}
Let $X$ be a premetric space. We say that $X$ is {\em complete}, if the  map $i:X\rightarrow\widehat{X}$ is onto.
\end{defn}

\section{Completion}
Having at hand the formal definition of completeness, we now introduce the natural notion of a completion of a premetric space. 

\begin{defn}
Let $X$ be a premetric space. We say that $Y$ is a {\em completion} of $X$, if $Y$ is complete, and there exists an isometric embedding $j:X\rightarrow Y$ with dense image.
\end{defn}

By Theorem \ref{Embedded-thm}, the existence of a completion of a premetric space $X$ will be immediately guaranteed once we show that  $(\widehat{X},\widehat{d})$ is complete. 
To show that  $\widehat{X}$ is complete, we prove  an extension theorem. We remark that the heart of the argument is this extension theorem,  the rest of the proof of the completeness of $\widehat{X}$ and the uniqueness of the completion  will be an algebraic manipulation of diagrams. 

\begin{thm}
\label{extension-thm}
Let $(A,d_A),(B,d_B)$ and $(C,d_C)$ be premetric spaces, $f:A\rightarrow B$  an isometric embedding, and  $h:A\rightarrow C$  an isometric embedding with dense image. Then there exists a unique isometric embedding  $g:C\rightarrow \widehat{B}$ such that the following diagram commutes
$$
\begin{tikzcd}
A\arrow{r}{f} \arrow{d}{h} &B \arrow{d}{i}\\ C \arrow{r}{g} &\widehat{B} 
\end{tikzcd}
$$

\begin{proof}
We define first  for each $c\in C$ and each non negative rational $q$,  the set 
$$
S_q^c =\{f(x)\in B:x\in A, \, d_C(h(x),c)\leq q/2\}.
$$
Thereby we set $S^c=\{S_q^c:q\in\mathbb{Q}^+\}$. We will show that  $S^c$ is a Cauchy family on $B$ for all $c\in C$. Indeed, since $h$ has dense image, given two nonnegative rationals $\mu$ and $\nu$,   there exists $x\in A$ such that $d_C(h(x),c)\leq\frac{1}{2} \min\{\mu, \nu\}$. Therefore  $f(x)\in S_\mu^c\cap S_\nu^c$ and  we have shown that $S_\mu^c\cap S_\nu^c\not=\emptyset$. On the other hand, to check the second condition of a Cauchy family, it suffices to show that $\diam S_\varepsilon^c\leq \varepsilon$ for all   $\varepsilon>0$. Let  $\varepsilon>0$,  then for all $x,y\in A$ such that $d_C(h(x),c)\leq \varepsilon/2$ and $d_C(h(y), c)\leq \varepsilon/2$, we have,  by the triangular inequality, that  $d_C(h(x),h(y))\leq \varepsilon$. Since $h$ and $f$ are isometric embeddings we also have $d_A(x,y)\leq \varepsilon$ and $d_B(f(x),f(y))\leq \varepsilon$. Thus $\diam S_\varepsilon^c\leq \varepsilon$  and $S^c$ is a Cauchy family.

Now define $g(c)=[S^c]$ for each $c\in C$. To prove that $g$ is an isometric embedding suppose $d_C(c,c')\leq p$. To see that  $\widehat{d}_B([S^c],[S^{c'}])\leq p$, fix $\varepsilon>0$,  since $\diam S_\varepsilon^c\leq \varepsilon$ and $\diam S^{c'}_\varepsilon\leq \varepsilon$,  it suffices to show that $\diam (S^c_{\varepsilon}\cup S^{c'}_{\varepsilon})\leq p+\varepsilon$. This follows immediately from the triangular inequality since for every $f(x)\in S^c_{\varepsilon}$ and $f(y)\in S^{c'}_{\varepsilon}$  we have  $d_B(f(x),f(y))\leq p+\varepsilon$. Conversely, suppose $\widehat{d}_B([S^c],[S^{c'}])\leq p$.  To see that $d_C(c,c')\leq p$  it suffices to show that $d_C(c, c')\leq p+\varepsilon$  for all $\varepsilon>0$. 
Then, let $\varepsilon>0$ and  choose  $\mu$ and $\nu$  such that $\diam (S^c_{\mu}\cup S^{c'}_{\nu})\leq p+\varepsilon/3$.  Since the image of $h$ is dense in $A$, there exist $a,a'\in A$ such that $d_C(h(a),c)\leq \frac{1}{3}\min \{\mu, \varepsilon\}$ and $d_C(h(a'),c')\leq \frac{1}{3}\min \{\nu, \varepsilon\}$. As $d_C(h(a),c)\leq \varepsilon/3$ and  $d_C(h(a'),c')\leq \varepsilon/3$, then $h(a)\in S^c_{\mu}$ and $h(a')\in S^{c'}_{\nu}$ and therefore  $d_C(h(a),h(a'))\leq p+ \varepsilon/3$. Hence by the  triangular inequality,  $d_C(c, c')\leq p+\varepsilon$  for all $\varepsilon>0$, and by upper continuity we conclude $d_C(c,c')\leq p$.

Finally observe that  for all $a\in A$ we have $f(a)\in S^{h(a)}_\varepsilon$ for all $\varepsilon>0$, thereby  $\diam (\{f(a)\}\cup S^{h(a)}_{\varepsilon})\leq \varepsilon$ for all  $\varepsilon>0$. Thus $\{\{f(a)\}\}$ is equivalent to $S^{h(a)}$, that is, $i(f(a))=(f(a))^*=[S^{h(a)}] =g(h(a))$. Therefore $g\circ h=i\circ f$ and the diagram commutes.
\end{proof}
\end{thm}

Finally, we show that $\widehat{X}$ is indeed a completion and it is unique up to isometry. 

\begin{thm} $(\widehat{X},\widehat{d})$ is a complete premetric space for every premetric space $(X,d)$ and thus every premetric space admits  a  completion which is unique up to isometry.
\end{thm}

\begin{proof}
We will show that $\widehat{X}$ is complete.  By Theorem \ref{Embedded-thm}, we have the isometric embeddings  $i:X\rightarrow \widehat{X}$ and $j:\widehat{X}\rightarrow\widehat{\widehat{X}}$ each one with dense image. Thus the map $k=j\circ i:X\rightarrow\widehat{\widehat{X}}$ has also dense image. Thereby we can apply the Theorem \ref{extension-thm} to find a map $l$ such that  the following diagram commutes
$$\begin{tikzcd}
X\arrow{r}{\mathrm{id}} \arrow{d}{k} &X \arrow{r}{i} \arrow{d}{i}&\widehat{X} \arrow{d}{j}
\\ 
\widehat{\widehat{X}}\arrow{r}{l} &\widehat{X} \arrow{r}{j} &\widehat{\widehat{X}} 
\end{tikzcd}$$
Contracting the diagram we have 

$$\begin{tikzcd}
X\arrow{r}{i} \arrow{d}{k} &\widehat{X} \arrow{d}{j}\\ \widehat{\widehat{X}} \arrow{r}{jl} &\widehat{\widehat{X}} 
\end{tikzcd}$$
By the uniqueness given by  Theorem \ref{extension-thm}  $j\circ l = \mathrm{id}$, and this proves that $j$ is onto, that is, that $\widehat{X}$ is complete.

To prove uniqueness, suppose  $Y$ is a complete premetric space and   $h:X\rightarrow Y$ is an isometric embedding with  dense image. Let $j:Y\to\widehat{Y}$ and $i:X\to\widehat{X}$ be  the natural embeddings. Since $Y$ is complete, then $j$ is a bijection. Then by Theorem \ref{extension-thm} there are  maps $u,v$ such that the following diagram commutes. 
$$
\begin{tikzcd}
X\arrow{r}{\mathrm{id}} \arrow{d}{h} &X \arrow{r}{jh} \arrow{d}{i}&\widehat{Y} \arrow{d}{j^{-1}}
\\ 
Y\arrow{r}{u} &\widehat{X} \arrow{r}{v} &Y 
\end{tikzcd}
$$
Contracting the diagram we have 
$$\begin{tikzcd}
X\arrow{r}{jh} \arrow{d}{h} &\widehat{Y} \arrow{d}{j^{-1}}\\ 
Y\arrow{r}{vu} &Y 
\end{tikzcd}$$
By uniqueness we have $v\circ u=\mathrm{id}$. Thus the map $u$ is an isometry between $Y$ and $\widehat{X}$. 
\end{proof}

Now we are going to show that  $(\widehat{X}_R, \widehat{d}_R)$ is R-complete.  Let $\widehat{X}_R$ be the quotient of all regular families under the equivalence relation defined in the introduction and $\widehat{d}_R$ be the premetric on $\widehat{X}_R$. The natural identification $i_R:X\to\widehat{X}_R$ is defined by $i_R(x)=[S^x]$, where $S^x=\{S_q:\; q\in \mathbb{Q}^+\}$ with  $S_q=\{x\}$ for all $q\in \mathbb{Q}^+$. Richman proved that $i$ is an isometric embedding with dense image (see \cite[Theorem 2.2]{Richman2008}). Let $\widehat{\widehat{X}}_R$ denote the quotient of all regular families of $(\widehat{X}_R,\widehat{d}_R)$ and  
$\widehat{i}_R:\widehat{X}_R\to \widehat{\widehat{X}}_R$ the natural isometric embedding. We will show that $\widehat{i}_R$ is onto.

We need the following lemma proved by Richman.

\begin{lmma}
\label{richlema} \cite[Lemma 2.1]{Richman2008}. Let $X$ be a premetric space and $S=\{S_q:q\in \mathbb{Q}^+\}$ a regular family on $X$. Then $d_R(i_R(x),[S])\leq q$ for any $x\in S_q$ and $q\in \mathbb{Q}^+$.
\end{lmma}

\begin{thm}
\label{the_last}
Let $X$ be a premetric space, then 
\begin{itemize}
\item[(i)] There exists an isometry  $\varphi:(\widehat{X}_R,\widehat{d}_R)\rightarrow(\widehat{X},\widehat{d})$.

\item[(ii)] The map $\widehat{i}_R:\widehat{X}_R\to \widehat{\widehat{X}}_R$ is onto.

\end{itemize}

\begin{proof}
(i) Since the natural map $i:X\rightarrow\widehat{X}_R$ is an isometric embedding with dense image. Then by Theorem \ref{extension-thm} there exists an isometric embedding $\varphi$ such that the next diagram commutes
$$
\begin{tikzcd}
X\arrow{r}{\mathrm{id}} \arrow{d}{i_R} &X\arrow{d}{j}\\ \widehat{X}_R \arrow{r}{\varphi} &\widehat{X}
\end{tikzcd}
$$
We prove now that $\varphi$ is onto. Let $s\in\widehat{X}$, we define for each positive rational $q$ the set 
\begin{eqnarray}
\label{cauchyregular}
S_q=\{x\in X:\widehat{d}(j(x),s)\leq q\}.
\end{eqnarray}
Since $j$ is an isometric embedding and using  the triangular inequality, it is easy to verify that $S=\{S_q:q\in\mathbb{Q}^+\}$ is a regular family. By Lemma \ref{richlema} $\widehat{d}_R([S],i_R(x))\leq q$, and thus $\widehat{d}(\varphi([S]),\varphi(i_R(x)))\leq q$, as $\varphi$ is an isometric embedding. By the commutativity of the diagram, we have that  $\widehat{d}(\varphi([S]),j(x))\leq q$. Since $x\in S_q$ we have also $\widehat{d}(j(x),s)\leq q$, so by triangular inequality on $\widehat{X}$ we conclude $\widehat{d}(\varphi([S]),s)\leq 2q$ for any $q\in\mathbb{Q}^+$, that is, $\varphi([S])=s$ and $\varphi$ is onto.

(ii)  Let $\widehat{j}:\widehat{X}\to \widehat{\widehat{X}}$ be the natural map, $\varphi:\widehat{X}_R\to \widehat{X}$ the map defined in part (i) and $\widehat{i}_R:\widehat{X}_R\to \widehat{\widehat{X}}_R$ the natural map. By Theorem \ref{extension-thm}, there is an isometric embedding $\psi:\widehat{\widehat{X}}_R\to \widehat{\widehat{X}}$ such that the following diagram commutes

$$
\begin{tikzcd}
\widehat{X}_R\arrow{r}{\varphi} \arrow{d}{\widehat{i}_R} &\widehat{X}\arrow{d}{\widehat{j}}\\ \widehat{\widehat{X}}_R \arrow{r}{\psi} &\widehat{\widehat{X}}
\end{tikzcd}
$$

Since $\widehat{j}$, $\varphi$ are onto maps, then it is easy to verify that $
\widehat{i}_R$ is also onto. 

\end{proof}
\end{thm}

Note that from the  proof of the  previous theorem we get that every Cauchy family is equivalent to a regular family (in the sense of Richman as defined in the introduction). And vice versa, every regular family is equivalent (in the sense of Richman) to a Cauchy family. In fact, the regular family  given by \eqref{cauchyregular}  is also a Cauchy family, and moreover, it is the maximal regular family  that  Richman \cite{Richman2008} showed exists in any equivalence class.

\bibliographystyle{plain}

\end{document}